\documentclass[11pt]{article}
\usepackage{amsmath,color,amsfonts,amsthm, amssymb}
\usepackage{amssymb,enumerate,verbatim}

\date{}
\title{Long monotone trails in random edge-labelings of random graphs}

\author{Omer Angel\thanks{University of British Columbia, supported in part by NSERC. Email: {\tt angel@math.ubc.ca}}
  \and Asaf Ferber \thanks{Massachusetts Institute of Technology. Department of Mathematics. Email: {\tt ferbera@mit.edu}. Research is partially supported by an NSF grant 6935855.}
  \and Benny Sudakov\thanks{Department of Mathematics, ETH, 8092 Zurich, Switzerland. Email: {\tt benjamin.sudakov@math.ethz.ch}
    Research supported in part by SNSF grant 200021-175573.}
  \and Vincent Tassion\thanks{Department of Mathematics, ETH, 8092 Zurich, Switzerland. Email: {\tt vincent.tassion@math.ethz.ch}
    Research supported in part by the NCCR SwissMAP, funded by the Swiss National Science Foundation.}}

\allowdisplaybreaks

\theoremstyle{plain}
\newtheorem{theorem}{Theorem}[section]
\newtheorem{lemma}[theorem]{Lemma}
\newtheorem{claim}[theorem]{Claim}

\newcommand{\eps}{\varepsilon}
\renewcommand{\P}{\mathbb{P}}
\newcommand{\E}{\mathbb{E}}

\begin{document}
\maketitle
\begin{abstract}
  Given a graph $G$ and a bijection $f : E(G)\rightarrow \{1, 2, \ldots,e(G)\}$, we say that a trail/path in $G$ is $f$-\emph{increasing} if the labels of consecutive edges of this trail/path form an increasing sequence. More than 40 years ago Chv\'atal and Koml\'os raised the question of providing the worst-case estimates of the length of the longest increasing trail/path over all edge orderings of $K_n$. The case of a trail was resolved by Graham and Kleitman, who proved that the answer is $n-1$, and the case of a path is still wide open. Recently Lavrov and Loh proposed to study the average case version of this problem in which the edge ordering is chosen uniformly at random. They conjectured (and it was proved by Martinsson) that such an ordering with high probability (whp) contains an increasing Hamilton path.

  In this paper we consider the random graph $G=G_{n,p}$ with an edge ordering chosen uniformly at random. In this setting we determine whp the asymptotics of the number of edges in the longest increasing trail. In particular we prove an average case version of the result of Graham and Kleitman, showing that the random edge ordering of $K_n$ has whp an increasing trail of length $(1-o(1))en$ and this is tight. We also obtain an asymptotically tight result for the length of the longest increasing path for random Erd\H{o}s-Renyi graphs with $p=o(1)$.

\end{abstract}

\section{Introduction}	
	
A \emph{trail} in a graph $G$ is a sequence of vertices $v_1,\ldots,v_t$ such that $v_i$ is adjacent to $v_{i+1}$ for all $i$, and no edge appears more than once. A \emph{path} is a trail where no vertex is repeated. Given a graph $G$ and a bijection $f : E(G)\rightarrow \{1, 2, \ldots,e(G)\}$, we say that a trail in $G$ whose edges (in consecutive order) are $(e_1, e_2, \ldots , e_k)$ is $f$-\emph{increasing} if the labels $f(e_1), f(e_2), \ldots, f(e_k)$ form an increasing sequence. Let $m(G)$ denote the largest integer $k$ for which every bijection $f:E(G)\rightarrow \{1,2,\ldots,e(G)\}$ gives an $f$-increasing path of length $k$, and $m^*(G)$ denote the largest integer $k$ for which every such $f$ gives an $f$-increasing trail of length $k$.

The problem of proving worst-case estimates for the length of the longest increasing trail/path in graphs goes back more than 40 years to Chv\'atal and Koml\'os \cite{CK}. In 1971 they asked to determine $m(K_n)$ and $m^*(K_n)$ for the complete graph on $n$ vertices $K_n$. For trails this problem was resolved by Graham and Kleitman, who showed that $m^*(K_n)=n-1$ unless $n\in \{3,5\}$ (in these cases $m^*(K_n)=n$). Graham and Kleitman \cite{GK} actually proved a lower bound for general graphs. Namely, they showed that every graph of average degree $d$ satisfies $m^*(G)\geq d$ (in particular, this implies $m^*(K_n)\ge n-1$). 

The problem of determining $m(G)$ and $m^*(G)$ for a general graph $G$ appears to be quite challenging. In particular, even in the case $G=K_n$, the lower and upper bounds for the length of the longest increasing path are still quite far apart. An old lower bound of Graham and Kleitman \cite{GK}, of order $\sqrt{n}$, was improved only in 2015 by Milans \cite{M} to $m(K_n)\geq n^{2/3}/\log^Cn$.
Very recently, a nearly linear lower bound $m(K_n)\geq n^{1-o(1)}$ was proved in \cite{B+}. For the upper bound, an old construction of Calderbank, Chung and Sturtevant from the 1980's \cite{C} gives $m(K_n)\leq (1+o(1))\frac n2$ and there were no improvements since then. There are also many results considering $m(G)$ and $m^*(G)$ for other graphs rather than $K_n$. The interested reader is referred to \cite{A,RSY,R,S,Y} and the references therein.

Rather than studying the worst case scenario, it is also natural to investigate the average case of the increasing trail/path problem, i.e., with respect to random edge labeling. Let $G$ be a graph on $n$ vertices and let $f: E(G)\rightarrow \{1,\ldots,e(G)\}$ be a bijection chosen uniformly at random. What can we say about the length of the longest $f$-increasing trail/path in $G$? This interesting question was raised by Lavrov and Loh \cite{LL}. They conjectured, and later Martinsson \cite{Mart} proved, that the uniform random edge ordering of $K_n$ whp (that is, with probability tending to $1$ as $n$ tends to infinity) contains an increasing path of length $n-1$, which is obviously best possible. What about the longest increasing trail in the random edge ordering of $K_n$? In this paper we answer this question.

Our results are more general and we consider increasing path/trail problems in the random graph setting. Let $G=G_{n,p}$ be a graph on $n$ vertices in which every pair $xy$ is an edge randomly and independently with probability $p$. Note that when $p=1$ we get the complete graph $K_n$. Expose the edges of $G=G_{n,p}$ and let $f:E(G)\rightarrow \{1,2,\ldots,e(G)\}$ be a random bijection. What can one say about the asymptotics of the length of the longest increasing path/trail for typical $G$ and $f$? To make the discussion a bit more formal, let $X_{k,p}$ and $Y_{k,p}$ be random variables which count the number of increasing paths and trails, respectively, of length $k$ in $G=G_{n,p}$ with respect to random edge ordering. It is is easy to check  that
$$\mathbb{E}[X_{k,p}]=\binom{n}{k+1}(k+1)!p^{k}\frac{1}{k!},$$
and
$$\mathbb{E}[Y_{k,p}]\leq n^{k+1}p^{k}\frac{1}{k!}.$$

Using Stirling's formula one can check that for any fixed $\varepsilon>0$ and $p=\omega(\log n/n)$, the expectation of $Y_{k,p}$ tends to $0$ for $k\geq (1+\varepsilon)enp$. By Markov's inequality, this implies that whp the longest increasing trail has length at most $(1+\varepsilon)enp$. Our first theorem shows that this bound is (asymptotically) tight.

\begin{theorem}\label{main}
  Let $\varepsilon>0$ be fixed, let $p=\omega(\log n/n)$, let $G=G_{n,p}$,  and let $f:E(G)\to \{1,\ldots,e(G)\}$ be a uniformly random edge ordering of $G$.
  Then, whp the longest increasing trail has length at least $(1-\varepsilon)enp$.
\end{theorem}

\noindent
When $p=1$ this theorem gives an analog for trails of the above mentioned result of Martinsson, showing that
the longest increasing trail in the random edge ordering of $K_n$ whp has length at least $(1-\varepsilon)en$. Compared with the result of
Graham and Kleitman it shows that a random ordering differs by a factor of $e$ from the worst case scenario.

For $p=o(1)$ our proof gives a bit more. In this regime we can actually produce not only a trail but a path of a similar length.  This gives the following result, which is tight, since the longest increasing path is not longer than the longest trail.

\begin{theorem}\label{main:paths}
  Let $\varepsilon>0$ be fixed, let $\log n/n \ll p \ll 1$, let $G=G_{n,p}$, and let $f:E(G)\to \{1,\ldots,e(G)\}$ be a uniformly random edge ordering of $G$.
  Then, whp the longest increasing path has length at least $(1-\varepsilon)enp$.
\end{theorem}

\noindent
Note that the above theorem does not cover the regime of $p$ being a constant. The case $p=1$ is covered by the main result in \cite{M}, and unfortunately, for $p=\Theta(1)$ our proof only gives paths of length around $(1-e^{-ep}-o(ep))n$. It would be interesting to derive an (asymptotically) optimal result also for constant $p$, and we leave this as an open problem.

Finally we remark that for the very sparse regime when $p=c/n, c>1$, it is easy to prove that the answer is $k=(1-o(1))(\log n/\log\log n)=\omega(np)$. Indeed, it is well known that whp $G_{n,p}$ contains a path of length $\Theta(n)$ (for more details, see e.g. \cite{Bol}). Expose $G$, fix such a path and cut it into $\Theta(n/k)$ edge-disjoint subpaths of length $k=\Theta(\log n/\log\log n)$ each. Now, by exposing $f$, the probability for each such subpath to become increasing is exactly $\frac{2}{k!}$ (there are two possible orientations) and the subpaths are mutually independent with respect to the property `being increasing'. Now, observe that as the expected number of increasing subpaths is $\Theta\left(\frac{n}{k\cdot k!}\right)=\omega(1)$, one can use Chernoff's bound (or the law of large numbers) to conclude that whp at least one such subpath is increasing. On the other hand,  if $k=(1+\varepsilon)(\log n/\log\log n)$ then $\mathbb{E}[Y_{k,p}]=o(1)$. Thus by Markov's inequality whp, there is no increasing trail (and hence no increasing path) of length $k$.

\section{Auxiliary results}

In this section we state (and prove) few lemmas that we need in the proofs of our main results. First, we show that a typical $G_{n,p}$ does not contain too many `short' cycles. All the results are asymptotic as $n$ tends to infinity.

\begin{lemma}\label{short cycles}
Let $p\gg 1/n$. Then, whp the number of cycles of length at most $k$ in $G_{n,p}$ is at most $(np)^{k+1}$.
\end{lemma}

\begin{proof}
Let $X_k$ denote the random variable counting the number of cycles of length at most $k$ in $G_{n,p}$. Clearly,
  \[
    \E[X_k]=\sum_{\ell=3}^k\binom{n}{\ell}\cdot \frac{(\ell-1)!}{2}p^{\ell}\leq \sum_{\ell=3}^k(np)^{\ell}\cdot \frac{1}{2\ell}\leq (np)^{k}.
  \]
  Since $p \gg 1/n$, the result now follows from Markov's inequality.
\end{proof}

For $0\le m\le \binom n 2$, let $G_{n,m}$ be a random graph on $n$ vertices with exactly $m$ edges, chosen uniformly at random among all such graphs. We make use of Lemma \ref{short cycles} in order to prove that $G_{n,m}$ typically contains a 'large' subgraph with `large' girth and `large' minimum degree.

\begin{lemma}\label{cor:high girth}
  Let $\log^{0.5}n/n\leq p\leq \log^2 n/n$ and $m=\binom{n}{2}p$. Then, the random graph $G_{n,m}$ whp contains a subgraph $H\subseteq G_{n,m}$ such that:
  \begin{enumerate}
    \item $|V(H)|\geq (1-o(1))n$,
    \item $\delta(H)\geq (1-o(1))np$, and
    \item $H$ has girth at least 
      $\frac{\log n}{2\log np}$.
  \end{enumerate}
\end{lemma}

\begin{proof}
  It is more convenient to work with the $G_{n,q}$ model. Let $q=p-p/\log^2n$, and observe that whp we have $e(G_{n,q})\leq m$ (this follows immediately from Chernoff's bounds). Therefore, one can easily couple $G_{n,q}$ as a subgraph of $G_{n,m}$ (by simply adding $m-e(G_{n,q})$ randomly selected edges to $G_{n,q}$). To prove the lemma, we show that $G_{n,q}$ whp contains a subgraph $H$ satisfying the required properties; then, whp $H\subseteq G_{n,q}\subseteq G_{n,m}$. Note that as $p=(1+o(1))q$, we can exchange them in our computations to obtain Properties 1.-3. with respect to $p$ instead of $q$, so let $G=G_{n,q}$.

  First, note that whp $e(G)=(1/2+o(1))n^2q$. Fix $k<\frac{\log n}{2\log np}$. It follows from Lemma \ref{short cycles} that whp one has at most $(nq)^{k+1}$ cycles of length at most $k$ in $G_{n,q}$. Therefore, by deleting one vertex from each such cycle we obtain a subgraph $G'$ satisfying Properties $1$ and $3$ of the lemma. Denote by $V'$ the set of deleted vertices. By construction, whp we have
  \begin{equation}
   \label{eq:4}
   |V'|\le (nq)^{k+1}=\exp((k+1)\log nq)\leq \exp(\frac12\log n+\log nq) \leq \sqrt{n}\log^2 n.
 \end{equation}
 Observe that as every subgraph of a graph of girth at least $k$ also has girth at least $k$, it is enough to show that there exists $H\subseteq G'$ with $\delta(H)\geq (1-o(1))nq$ and with $|V(H)|\geq (1-o(1))n$. To do so, fix $\varepsilon>0$ and consider the following process. Let $V''$ be the set of all vertices in $G$ with degree at most $(1-\varepsilon)nq$ and let $V_0=V' \cup V''$. Now, as long as there exists a vertex $v$ in $V(G)\setminus V_i$ with degree at least $\varepsilon nq$ into $V_i$, do the following. Let $v$ be such a vertex, and define $V_{i+1}:=V_i\cup \{v\}$. We show that this process must terminate after at most (say) $\ell=n/\log n$ iteration. To this end let us note that by Chernoff's bounds and Markov's inequality, one can easily obtain that whp
  \begin{equation}
|V''|=n\cdot \exp(-\Theta(nq))\leq \frac{n}{e^{\log^{0.4}n}}.\label{eq:6}
\end{equation}
Using \eqref{eq:4} and \eqref{eq:6}, we see that after $\ell$ steps we obtain a set $V_\ell$ with at most $|V_0|+\ell\leq 2\ell$ vertices, and with at least $\varepsilon nq\ell$ edges. We show that this is impossible in $G_{n,q}$. Indeed, given a subset $X\subseteq V(G_{n,q})$ of size $\ell \leq |X|\leq 2\ell$, the number of edges in $G_{n,q}[X]$ is distributed as $\text{Bin}(\binom{|X|}{2},q)$. Therefore, the probability to have at least $\varepsilon nq\ell$ edges in $G_{n,q}[X]$ is at most
  $$\binom{|X|^2}{\varepsilon nq\ell}q^{\varepsilon nq\ell}\leq \left(\frac{e|X|^2q}{\varepsilon nq\ell}\right)^{\varepsilon nq\ell}\leq \left(\frac{\ell}{n}\right)^{\varepsilon nq\ell/2}\leq e^{-0.5\varepsilon nq\ell\log \frac{n}{\ell}}.$$

  Now, by applying the union bound to all subsets of sizes between $\ell$ to $2\ell$, as there are at most $2\ell \binom{n}{2\ell}= e^{O(\ell \log \frac{n}{\ell})}$ of them, we obtain that there is no such subset $V_{\ell}$.

  In order to complete the proof, let $s$ be the last step of the above process, and let $H:=G'\setminus V_{s}$. Then we can easily check that whp $|V(H)| \ge n(1-3/\log n)$,  $\delta(H)\ge (1-2\varepsilon)qn=(1-2\varepsilon-o(1))pn$ and $H$ has girth larger than $k$ (since it is a subgraph of $G'$).
\end{proof}

The next lemma, which might be of independent interest, studies increasing paths in random edge labelings of trees. Before stating it, we need to introduce some notation. Let $T^k_D$ be the rooted $D$-ary tree with $k$ levels (that is, there is a root $r$ of degree $D$, and each of its neighbors has $D$ descendants and so on for $k$ levels, where the last level are leaves). Here we prove an asymptotically best possible dependency between $k$ and $D$ for which a random labeling of the edges of $T_D^k$ whp has an increasing path from the root to some leaf. Our proof relies on standard methods in the study of branching random walks. More precisely, we apply a  second moment method and use a truncation argument similar to the one appearing e.g.\@ in \cite{DR,McD}.
Here, the terminology whp refers to the asymptotic behavior as $D$ tends to infinity.

\begin{lemma}
  \label{monotone path in one step}
  Fix $\varepsilon>0$, and let  $k\leq (1-\varepsilon)eD$. Then, a random uniform labeling of $E(T^k_D)$ whp results in an increasing path from $r$ to some leaf.
\end{lemma}

\begin{proof} Note that it is enough to consider the case where $\varepsilon$ is small (for larger values, we actually prove a stronger statement). It will be convenient for us to consider a random bijection $f:E(T_D^k)\rightarrow \{1,\ldots,e(T_D^k)\}$ as follows: for every edges $e\in E(T_D^k)$ we assign a random variable $X(e)$, uniform in $[0,1]$, where all the variables are independent. With probability $1$ all the labels are distinct and therefore the $X(e)$'s naturally define $f$ by assigning the labels $\{1,\ldots,e(T_D^k)\}$ to the edges according to the natural ordering of the $X(e)$'s. Clearly, the obtained $f$ is a uniformly chosen bijection.

Let us first observe that the constant $e$ in the lemma is best possible. Indeed, the expected number of increasing paths from the root to some leaf is $D^k\frac{1}{k!}\approx \left(\frac{eD}{k}\right)^k$, and this clearly goes to $0$ whenever $k\geq (1+\varepsilon)eD$.

  Now, consider the number $Y$ of paths from the root to some leaf of $T^{k}_D$ along which the labels are increasing and satisfy $X(e)<1-\eps/2$.
  In order to prove the result it suffices to show that there exists a constant $c>0$ (that may depend on $\varepsilon$) such that  for $k\le (1-\varepsilon) eD$,
  \begin{equation}
    \label{eq:1}
    \mathbb P[Y\ge 1]\ge \frac{c}{k^{3/2}}.
  \end{equation}
  Indeed, if we replace $Y$ by a random variable  $Y'$ which counts the number of paths from the root to some leaf of $T^{k}_D$ along which the labels are increasing and satisfy $X(e)>\eps/2$, we obtain that $Y'$ has the exact same distribution as $Y$. Moreover, whp\ the root of $T^{k+2}_D$ has at least $\eps^2 D^2/9$ paths of length $2$ with labels $0<a<b<\eps/2$ (the expected number of such paths is $(D^2/2)(\eps/2)^2$).
  Then, the estimate above shows that each of these short paths has probability larger than $c/{k^{3/2}}$ to be extendable into an increasing path to some leaf of $T^{k+2}_D$. Since these trees are disjoint, by independence, we obtain that whp\ there exists  an increasing path from $r$ to some leaf of $T^{k+2}_D$.   

  Let us now turn to the proof of \eqref{eq:1}.
  Applying the second-moment method to $Y$ naively fails, since if we condition on two paths with several common edges from the root to some leaves to be increasing, the labels along the common edges will be very different from two independent paths conditioned to be increasing.
  This leads to a dominant contribution to the second moment from paths which are very different from typical increasing paths (it is worth mentioning that the naive approach gives a constant $2$ instead of $e$ in the lemma, which is already non-trivial).

  To overcome this problem, we introduce below the notion of `good paths' which are increasing paths with some additional restrictions on the labels.
  Let $(X_1,\ldots,X_k)$ be the labels along a fixed path from the root to some leaf of $T^k_D$.
  For $\delta=\varepsilon/2$, say that this path is \emph{good} if the labels satisfy
  \begin{enumerate}[$1.$]
  \item Monotonicity: $X_1\le\cdots\le X_k$,
  \item The last label satisfies $1-\delta-\frac1k \le X_k\le 1-\delta$, and
  \item A lower bound: for every $1\le i\le k$, $X_i\ge \frac i k X_k$.
  \end{enumerate}
  We will apply a second moment method to show that the number $Z$ of good paths is positive with probability larger than $c/k^{3/2}$. The result will then follow from the fact that $Y\geq Z$ (which holds deterministically).

  We begin with the computation of the probability of a fixed path to be good. We use a standard trick which is based on Spitzer's Lemma (see \cite{Spitzer}). Consider the labels $(X_1,\ldots,X_k)$ along the path to some fixed leaf of $T^k_D$. Conditional on the event
$X_1\le\cdots\le X_k$ and on the value of $X_k$, the law of the increments $I_i=X_{i+1}-X_i$ (where $X_0=0$) 
is invariant under cyclic permutations. In other words, under $\P[\cdot | X_1\le\cdots\le X_k,X_k]$ we have
\begin{equation}
 (I_1,\ldots,I_k) \overset{\text{law}}{=} (I_{c(1)},\ldots,I_{c(k)})
\end{equation}
for every cyclic permutation $c$ of $\{1,\ldots,k\}$. Now it is easy to show (see \cite{Spitzer}) that for any outcome
there exists exactly one cyclic ordering of these increments with $\sum_{i\leq j}I_{i}\geq \frac{j}{k}X_k$ for all $j$.
Hence, the conditional probability for $3.$ to hold is $1/k$.

  From the discussion above we have
  \begin{align}
    \P[ \text{ the path is good}]&=\frac{1}{k} \P[X_1 \le \dots \le X_k,\, 1-\delta-\tfrac1k \le X_k\le 1-\delta]\notag\\
    &=\frac{1}{k}\cdot \frac1{k!}\left[(1-\delta)^k-(1-\delta-\tfrac 1k)^k\right]. \label{eq:2}
  \end{align}
  Note that $(1-\delta)^k-(1-\delta-1/k)^k\leq (1-\delta)^k$ and that for a small enough $\delta$ we have
  $$(1-\delta)^k-(1-\delta-1/k)^k\geq (1-\delta)^k(1-1/(1-\delta)k)^k\geq (1-\delta)^k(1-e^{-(1+o(1))/(1-\delta)})\geq (1-\delta)^k/2.$$ Therefore, combining these estimates with \eqref{eq:2} we find that
  \begin{equation}
    \label{eq:9}
    \frac {1} {2k \cdot k!} (1-\delta)^k \le \P[\text{ the path is good}] \le\frac 1 {k \cdot k!}(1-\delta)^k.
  \end{equation}
  
  Since the expectation of $Z$ is equal to $D^k\cdot \P[E_k]$, we obtain, using Stirling's formula, that
  \begin{equation}
    \label{eq:3}
    \E[Z]\ge \frac {D^k}{2k\cdot k!}  (1-\delta)^k \geq C Q^k k^{-3/2},
  \end{equation}
  where $Q = \frac{De(1-\delta)}{k}$, and $C$ is some absolute constant.

  We now bound the second moment of $Z$.
  Consider two paths in $T^k_D$, say $(e_1,\dots,e_k)$ and $(h_1,\dots,h_k)$.
  Suppose the two paths have $k-i$ common edges, so that $e_{k-i}=h_{k-i}$ is the last common edge.
  If the $e$ path is good, then $X(e_{k-i})\ge \frac {k-i}k X(e_k)\ge \frac {k-i}k(1-\delta-\tfrac1k)$.
  Conditionally on the $e$ path being good, the $h$ path is good with probability smaller than
  \begin{equation*}
    \P\left[\frac {k-i} k(1-\delta-\tfrac 1k)\le X(h_{k-i+1})\le \cdots\le X(h_k)\le 1-\delta \right].
  \end{equation*}
  The variables are increasing with probability $1/i!$ and are all in the necessary interval with probability at most
  $$\left((1-\delta)-(1-\frac{i}{k})(1-\delta-\frac{1}{k})\right)^{i}\leq \left(\frac{i(1-\delta)+1}{k}\right)^i\leq e^{1/(1-\delta)} \left(\frac{i(1-\delta)}{k}\right)^i.$$
 Hence,
  \[
    \P[\text{$h$ is good} | \text{$e$ is good}] \leq
    e^{1/(1-\delta)}\frac{1}{i!} \left(\frac{i(1-\delta)}{k}\right)^i \leq
    e^{1/(1-\delta)} \left(\frac{e(1-\delta)}{k}\right)^i.
  \]
  The number of pairs of paths with $k-i$ common edges is bounded by $D^{k+i}$, and so
  \begin{align}
    \E[Z^2]
    &\leq \sum_{i=0}^k e^{1/(1-\delta)} D^{k+i} \P[E_k] \left(\frac{e(1-\delta)}{k}\right)^i\\
    &= e^{1/(1-\delta)} \E[Z] \sum_{i=0}^k Q^i.
  \end{align}

  Recall that we set $\delta=\eps/2$. In this case we have that $Q \geq \frac{1-\delta}{1-\eps} > 1$ which implies that the sum is within a constant factor with its last term. Using \eqref{eq:3} we obtain
  \begin{equation}
    \label{eq:E_Z2_bound}
    \E[Z^2] =O(\E[Z]^2 k^{3/2})
  \end{equation}
  (with constant depending on $\eps$).
  The Cauchy-Schwarz inequality implies that for some $C'>0$ we have
  \begin{equation}
    \label{eq:5}
    \mathbb P[Z\ge1]\ge \frac{\E[Z]^2}{\E[Z^2]}\ge \frac 1{C'k^{3/2}}.
  \end{equation}
  Since $Y\ge Z$ deterministically, the equation above trivially implies Equation~\eqref{eq:1}.
\end{proof}

\paragraph{Remark.}
The bound of $\frac{c}{k^{3/2}}$ in \eqref{eq:1} can be improved to $c/k$ by applying Stirling to $i!$ above.

\section{Proof of Theorem \ref{main}}

In this section we prove our main result. As noted in the introduction (before the statement of the theorem), the upper bound follows by a simple union bound, so we only need to address the lower bound.
The main idea is to partition the graph into several subgraphs $G_i$ with consecutive values of edge weights.  In each of these $G_i$, we find with high probability many reasonably long increasing trails.  In order to combine these, we leave aside a smaller number of the edges between the edges of $G_i$ and $G_{i+1}$.  We then argue that with high probability the end of any trail in $G_i$ is connected to the beginning of some trail in $G_{i+1}$ by one of these edges. This allows us to stitch together the individual trails to a single long trail.  We proceed to make this precise.

\medskip

Fix $\varepsilon>0$ and $p=\omega(\log n/n)$. Let $G=G_{n,p}$ and let $f$ be a random bijection as in the assumptions of the theorem. Our goal is to show that whp $G$ contains an $f$-increasing path of length at least $(1-\varepsilon)enp$. Note that we may further assume that $p\leq 1-\varepsilon/10$. Indeed, assume $p$ is larger, and replace it by $p'=1-\varepsilon/10$. This gives us an increasing trail of length at least $(1-\varepsilon)enp'\geq (1-2\varepsilon)np$, and by re-scaling we obtain the result.

Before describing our algorithm, we need some preparation. First, expose the number of edges $m=e(G_{n,p})$ (but not the edges themselves). Note that whp we have $m=(1/2+o(1))n^2p$. Second, note that we can choose $t:=t(n)$, $a:=a(n)$, and $b:=b(n)$ such that $t:=(1-o(1))\frac{np}{\log^{0.5}n}$, $a=n\log\log n$, $b=(1-o(1)) \frac n 2\log^{0.5}n$, and $(a+b)t=m$. Partition $[m]$ into $2t$ consecutive and disjoint intervals $[m]=I_1\cup J_1\cup I_2\cup J_2\ldots \cup I_t\cup J_t$ in such a way that $|I_i|=a$ and $|J_i|=b$ for all $i$. For each $i$, let $H_i$ and $G_i$ be the subgraphs of $G_{n,p}$ induced by the edges with labels from $I_i$ and $J_i$, respectively. Clearly, $G_i\overset{\text{law}}=G_{n,a}$ and $H_i\overset{\text{law}}=G_{n,b}$ for all $i$, where for a fixed integer $x$, $G_{n,x}$ is a graph on $n$ vertices with exactly $x$ edges, chosen uniformly at random among all such possible graphs (note that these graphs have disjoint edge sets, and so are not independent).

Now we are ready to describe our algorithm. The algorithm consists of $t$ \emph{rounds}, where each round consists of two \emph{steps}, one of which is performed within $G_i$ and the other within $H_i$. After each round $i$ we obtain an increasing trail $T_i$ for which:
\begin{enumerate}
\item $T_{i-1}\subseteq T_{i}$ (that is, $T_{i-1}$ is an initial segment of $T_i$); and
\item $E(T_i)\subseteq \bigcup_{j=1}^i \left(E(G_j)\cup E(H_j)\right)$; and
\item  either $T_i=T_{i-1}$, and in this case we consider the $i$th round as a \emph{failure}, or the \emph{length} of $T_i$, denoted as $\ell(T_i)$, satisfies $\ell(T_i)\geq \ell(T_{i-1})+s$, where $s$ will be determined below.
\end{enumerate}

Our goal is to prove that whp $\ell(T_t)\geq (1-\varepsilon)enp$, which is equivalent to
$$\ell(T_0)+\sum_{i=1}^t \left(\ell(T_i)- \ell(T_{i-1})\right)\geq (1-\varepsilon)enp.$$

Initially, $T_0=\emptyset$. Suppose that we are at the beginning of round $i\geq 1$, and $T:=T_{i-1}=v_1\ldots v_x$ satisfies the three properties as defined above. Expose all the edges of $G_i$ without assigning them with the exact labels of $f$ (recall that all its labels are taken from the interval $J_i$). By Lemma \ref{cor:high girth} we know that whp there exists a subgraph $G'_i\subseteq G_i$ with $|V(G'_i)|\geq (1-o(1))n$, $\delta(G'_i)\geq d:=(1-\varepsilon/2)2a/n$ and with girth at least (say) $k=\log^{0.9}n$. Therefore, all the vertices in $G'_i$ serve as roots of some $d$-ary tree of depth $k$. Note that if such a $G'$ does not exists, then this round is a failure and we set $T_i=T$.

Now, exposing the exact values of $f$ on $E(G_i')$, by Lemma \ref{monotone path in one step} and Markov's inequality we obtain that whp there exists a subset $U_i$ of vertices of size $(1-o(1))n$, such that for all $u\in U_i$ there exists an $f$-increasing path of length $(1-\varepsilon/2)ed$ with $u$ as its starting point. Again, if there is no such set then we declare the $i$th round as a failure and set $T_i=T$. As all its labels are taken from $J_i$, it follows that all its labels are larger than the labels of $T$. Finally, expose the edges (and labels) of $H_i$. In the following claim we show that whp there exists a vertex $u\in U_i$ for which $v_xu\in E(H_i)$ (if not we declare this round as a failure). Suppose it is true, and let $Q$ denote an $f$-increasing trail in $G'_i$ with $u$ as its starting point. Define $T_{i}=v_1\ldots v_xuQ$ and observe that $T$ is an $f$-increasing trail of length at least $x+(1-\varepsilon/2)ed$ which extends $T_{i-1}$. Therefore, we can choose $s=(1-\varepsilon/2)ed$ in order to satisfy property $2$. In the case $T_{i-1}=\emptyset$, the ``gluing'' step is useless and we can simply set $T_i=Q$ starting from an arbitrary point.
 \begin{claim}\label{claim1}
   With high probability $H_i$ contains an edge from $v_x$ to $U_i$.
 \end{claim}
 \begin{proof} Note that the edges of $H_i$ are being chosen uniformly at random among the non-edges of previous $G_j$'s and $H_j$'s. Moreover, (recall that we assume $p\leq 1-\varepsilon/10$) as whp we have $d_G(v_x)\leq (1+o(1))np\leq (1-\varepsilon/100)n$ and $|U_i|-d_G(v_x)\geq \varepsilon n/200$, it follows that there are at least $\varepsilon n/200$ `free' edges between $v_x$ and $U_i$. Recall that we work in $H_i=G_{n,b}$ so the probability for not having an edge between $v_x$ and $U_i$ is at most
 $$\frac{\binom{\binom{n}{2}-\varepsilon n/200}{b}}{\binom{\binom{n}{2}}{b}}\leq
 \left(1-\frac{\varepsilon n/200}{{n \choose 2}}\right)^b=e^{-\Theta(b/n)}=o(1),$$
 where we use that $b=\omega(n)$ and that for any $p>r>q$, ${p-q \choose r}/{p \choose r}\leq (\frac{p-q}{p})^r$.
 \end{proof}
To summarize, by Markov's inequality, whp there are at most $o(t)$ rounds which are considered as failures. Therefore, in at least $t-o(t)$ rounds, the length of the current trail $T_i$ extends by $s$. Moreover, as $s\geq (1-\varepsilon/2) 2a/n$ we obtain that whp $\ell(T_t)\geq (1-\varepsilon)e2at/n\geq (1-2\varepsilon)enp$. This completes the proof.
\hfill $\Box$

\section{Proof of Theorem \ref{main:paths}}

The proof of Theorem \ref{main:paths} is more or less identical to the proof of Theorem \ref{main}. The only difference is that in order to obtain a path (as opposed to a trail), we need to restrict ourselves to trees which are vertex disjoint from our `current' path $P_{i-1}$ (which plays the role of $T_{i-1}$ in the proof of Theorem \ref{main}). Here we are using the fact that $p=o(1)$, so the total length of the path that we are trying to construct is at most $enp=o(n)$, and therefore, in each step we still have $(1-o(1))n$ `available' vertices to work with (that is, vertices which are not used in our current path). Under this restriction, the rest of the strategy and the calculations are basically the same as in Theorem \ref{main} so we omit the details. \hfill $\Box$

{\bf Acknowledgments} The authors are grateful to Oliver Riordan and to the anonymous referees for many valuable comments.


\begin{thebibliography}{10}

\bibitem{A} N. Alon, \emph{Problems and results in extremal combinatorics. I.} Discrete Math., 273(1-3) (2003):31–-5.

\bibitem{Bol}B. Bollob\'as, Random graphs, Cambridge Stud. Adv. Math. 73, Cambridge University Press,
Cambridge (2001).

\bibitem{B+}
M. Buci\'c, M. Kwan, A. Pokrovskiy, B. Sudakov, T. Tran and A. Wagner, Nearly-linear monotone paths in edge-ordered graphs,
 arXiv preprint arXiv:1809.01468 (2018).

\bibitem{C} A. R. Calderbank, F. R. Chung, and D. G. Sturtevant, \emph{Increasing sequences with nonzero block sums and increasing paths in edge-ordered graphs}, Discrete mathematics 50 (1984): 15--28.
\bibitem{CK} V. Chv\'atal and J. Koml\'os, \emph{Some combinatorial theorems on monotonicity}, Canad. Math.
Bull., 14 (1971):151–-157.
\bibitem{DR}
Luc Devroye and Bruce Reed.
\newblock {\em On the variance of the height of random binary search trees.}
\newblock  SIAM J. Comput., 24(6):1157--1162, 1995.
\bibitem{GK} R. L. Graham and D. J. Kleitman, \emph{Increasing paths in edge ordered graphs}, Periodica Mathematica Hungarica 3, no. 1-2 (1973): 141--148.

\bibitem{LL} M. Lavrov, and P.-S. Loh, {\em Increasing Hamiltonian paths in random edge orderings}, Random Structures \& Algorithms 48.3 (2016): 588--611.

\bibitem{Mart} A. Martinsson, {\em Most edge-orderings of $ K_n $ have maximal altitude}, arXiv preprint arXiv:1605.07204 (2016).
\bibitem{McD}
Colin McDiarmid.
\newblock {\em Minimal positions in a branching random walk}.
\newblock Ann. Appl. Probab., 5(1):128--139, 1995.
\bibitem{M} K. G. Milans, \emph{Monotone Paths in Dense Edge-Ordered Graphs}, arXiv preprint arXiv:1509.02143 (2015).
\bibitem{RSY} Y. Roditty, B. Shoham, and R. Yuster, \emph{Monotone paths in edge-ordered sparse graphs}, Discrete
Math., 226(1-3) (2001):411–417.
\bibitem{R}  V. R\"odl, Master’s thesis, Charles University, 1973.
\bibitem{S} J. De Silva, T. Molla, F. Pfender, T. Retter, and M. Tait, \emph{Increasing Paths in Edge-Ordered Graphs: The Hypercube and Random Graph}, The Electronic Journal of Combinatorics 23, no. 2 (2016): P2--15.
\bibitem{Spitzer} F. Spitzer, \emph{A combinatorial lemma and its application to probability theory}. Transactions of the American Mathematical Society, 82(2) (1956), 323-339.
\bibitem{Y} R. Yuster, \emph{Large monotone paths in graphs with bounded degree}, Graphs Combin., 17(3) (2001):579–-
587.

\end{thebibliography}
\end{document}